\documentclass[12pt]{amsart}

\usepackage[english]{babel}
\usepackage[utf8]{inputenc}

\usepackage[T1]{fontenc}
\usepackage{amsmath,amssymb}
\usepackage{amsfonts}
\usepackage{amsthm}
\usepackage{delarray,graphicx}
\usepackage{comment,color}

\theoremstyle{plain}
\newtheorem{theorem}{Theorem}[section]
\newtheorem{lemma}[theorem]{Lemma}
\newtheorem{proposition}[theorem]{Proposition}
\newtheorem{fact}{Fact}
\newtheorem{cor}{Corollary}

\newtheorem*{theorem*}{Theorem}
\newtheorem*{appl*}{Application}

\theoremstyle{definition}
\newtheorem{definition}{Definition}[section]

\theoremstyle{remark}
\newtheorem*{remark}{Remark}

\newcommand\dr{\mathbf}
\renewcommand\d{\textrm}
\renewcommand\phi{\varphi}

\newcommand\st{\textrm{ such that }}
\newcommand\Q{{\mathbf Q}}
\newcommand\bk{{\mathbf k}}
\newcommand\T{{\mathcal T}}
\newcommand\D{{\mathbf D}}
\newcommand\R{{\mathbf R}}
\newcommand\Z{{\mathbf Z}}

\DeclareMathOperator\Stab{Stab}
\DeclareMathOperator\Card{Card}
\renewcommand\P{{\mathbf P}}
\renewcommand\S{{\mathbf S}}
\DeclareMathOperator\PSL{PSL}
\DeclareMathOperator\GL{GL}
\DeclareMathOperator\SL{SL}
\DeclareMathOperator\SO{SO}
\DeclareMathOperator\PO{PO}
\DeclareMathOperator\PGL{PGL}
\DeclareMathOperator\Ad{Ad}
\DeclareMathOperator\Hull{Hull}
\DeclareMathOperator\Isom{Isom}
\renewcommand\H{\mathcal {H}}
\newcommand\C{\mathcal {C}}

\title[Yet another $p$-adic hyperbolic disc]{Yet another $p$-adic hyperbolic
disc : Hilbert distance for $p$-adic fields}

\author{Antonin Guilloux}
\address{IMJ, Paris 6, 4 Place Jussieu, 75005 Paris. aguillou@math.jussieu.fr}

\begin{document}

\maketitle


We describe in this paper a geometric construction in the projective $p$-adic
plane that gives, together with a suitable notion of $p$-adic convexity, some
open subsets of $\mathbf P_2(\Q_p)$ naturally endowed with a ``Hilbert''
distance and a transitive action of $\PGL(2,\dr Q_p)$ by isometries. These open
sets are natural analogues of the hyperbolic disc, more precisely of Klein's
projective model. But, unlike the real case, there is not only one such
hyperbolic disc. Indeed, we find three of them if $p$ is odd (and seven if
$p=2$).

Let us stress out that neither the usual notion of convexity nor that of
connectedness as known for the real case are meaningful in the $p$-adic case.
Thus, there will be a rephrasing game for the definitions of real convexity
until we reach a formulation suitable  for other local fields. It will lead us
to a definition of $p$-adic convexity by duality. Although we will not recover
the beautiful behaviour of real convexity, we will still be able to define
the most important tool for our goals, namely the Hilbert distance. 

We construct our analogues of the hyperbolic disc (once again, via the projective model of the hyperbolic plane) in a quite geometric, even naive, way. Our construction gives $2$-dimensional objects over $\Q_p$. It is very different, in spirit and in facts, of Drinfeld $p$-adic hyperbolic plane \cite{drinfeld}. The possible relations between the two objects remain still unexplored. Another object often viewed as an analogue of the hyperbolic disc is the tree of $\PGL(2,\Q_p)$ \cite{serre}. We explore the relations between our discs and this tree, constructing a natural quasi-isometric projection from the discs to the tree. Eventually we explore the transformation groups of our discs. And, whereas the transformation group of the tree is huge, we prove that only $\PGL(2,\Q_p)$ acts on the discs preserving the convex structure.


\tableofcontents

\section{Looking for a two-sheeted hyperboloid}\label{sec:hyperboloid}

In this section we follow the usual construction of the hyperboloid model for the hyperbolic space but over a local field $\bk$. The point is to analyse the properties of squares in $\bk$. We are mainly interested in the action of $\PGL(2,\bk)$ on $\bk^3$ via the adjoint representation: 
$$ \Ad \: : \: \PGL(2,\bk) \to \SL(3,\bk). $$
It is an isomorphism with the group $\SO(Q)$, where $Q$ denotes the quadratic form on $\bk^3$ given by $Q(x,y,z)=xz-y^2$. 
This section describes the level surfaces of $Q$ in $\bk^3\setminus\{0\}$.
Each of them is a single $\PGL(2,k)$-orbit (by Witt's theorem). We look at
their decomposition into $\PSL(2,k)$-orbits, like the two-sheeted hyperboloid
in the real case. These level sets can be of one of the three following types:
\begin{itemize}
\item  the isotropic cone, which is a finite union of $\PSL(2,\bk)$-orbits,
\item  a single $\PSL(2,\bk)$-orbit (a one-sheeted hyperboloid),
\item  the union of two $\PSL(2,\bk)$-orbits (a two-sheeted hyperboloid). 
\end{itemize}
The latter case is the most interesting for our concerns. A hyperbolic disc will
be, in some sense, a positive cone on sheets of these hyperboloids (see section
\ref{ssec:duality}). We will achieve its construction in a fully elementary way.
But, due to the lack of connectedness argument, some proofs rely on direct
algebraic computations ; we will postpone it to an annex. Keeping in
mind the real counterpart of these results should guide the intuition. We begin
by recalling some general facts about orthogonal groups and gradually focus on
the orthogonal group $\SO(Q)$ described above.

\subsection{Special orthogonal groups and level sets}

Consider $\bk$ a field of characteristic different from $2$, an integer $n\geq
1$, a $n+1$-dimensional $\bk$-vector space  $V$ and a quadratic form $q$ on $V$.
Then Witt's theorem \cite[42.F]{o'meara} implies that the special orthogonal
group $\SO(q)$ acts transitively on each level set of $q$ in $V\setminus \{0\}$.

Consider the form $Q(x_0,\ldots,x_n)=x_0x_n - x_1^2-\ldots -x_{n-1}^2$ on
$V=\bk^{n+1}$. The isotropic cone $\mathcal C$ of $Q$, i.e. the set of vectors
$v$ with $Q(v)=0$, decomposes into two orbits under the action of $SO(Q)$: the
singleton $\{0\}$ and its complement.

Let us now assume moreover that $\bk$ has the following property:
$x_1^2+\ldots +x_{n-1}^2$ is a non-zero  square for any non-zero
vector $(x_1,\ldots,x_{n-1})\in\bk^{n-1}\setminus\{0\}$. This holds for any $n$
if $\bk=\R$ and for any field $\bk$ if $n=2$. We prove then that the isotropic
cone contains a positive semi-cone defined by the fact that $x_0$ is a square in
$\bk$ ("positive" is an analogy with the real case in which the squares are the
positive numbers). This semi-cone is stabilized by an explicit normal subgroup
of $\SO(Q)$ which is of finite index for any local field $\bk$. The real case
tells us a useful interpretation for this finite index subgroup: it becomes the
connected component $\SO^o(Q)$ - in the case $\bk=\R$, the quadratic form $Q$
has signature $(1,n)$. Let $(\bk^*)^2$ be the set of squares (invertible)
elements in $\bk^*$ and  $\bar \alpha=\alpha (\bk^*)^2$ the class modulo the
squares of an element $\alpha \in \bk^*$. We get the following proposition:

\begin{proposition}\label{prop:isotropic}
Consider the form $Q$ over a field $\bk$ (char$(\bk)\neq 2$) where $x_1^2+\cdots
+x_{n-1}^2$ is a non-zero square for any non-zero
vector $(x_1,\ldots,x_{n-1})\in\bk^{n-1}\setminus\{0\}$. For any class $\bar
\alpha$ in
$\bk^*/(\bk^*)^2$, define the semi-cones:
$$\mathcal C_{\bar \alpha}=\{(x_0,\ldots,x_n)\in \mathcal C\setminus\{0\}
\st
x_0\textrm{ and }x_n\textrm{ belong to }\bar\alpha\cup \{0\}\}.$$

Then $\mathcal C\setminus\{0\}$ decomposes into the disjoint union of the
semi-cones over the elements of $\bk^*/(\bk^*)^2$. Moreover $\SO(Q)$ acts by
permutations on the set of semi-cones and we have an isomorphism
$$SO(Q)/\Stab(\mathcal C_{\bar 1}) \simeq \bk^*/(\bk^*)^2.$$
\end{proposition}

Before proving the proposition, let us describe an avatar of Iwazawa
decomposition of the group $\SO(Q)$. Let $Q'$ be the quadratic form
$x_1^2+\ldots +x_{n-1}^2$. Consider the three following subgroups of
$\SL(n+1,\bk)$:
\begin{itemize}
 \item $N^+=\begin{pmatrix} 1 & 2 {}^t w A & Q'(w) \\ 0 & A & w \\ 0 & 0 & 1
\end{pmatrix}$ for $A \in \SO(Q')$ and $w\in \bk^{n-1}$.
 \item $N^-=\begin{pmatrix} 1 & 0 & 0 \\ v & B & 0 \\ Q'(v) &2{}^t v B & 1
\end{pmatrix}$ for $B \in \SO(Q')$ and $v\in \bk^{n-1}$.
 \item $H=\begin{pmatrix} x & 0 & 0 \\ 0 & Id & 0 \\ 0&0&\frac{1}{x}
\end{pmatrix}$ ($x\in\bk^*$).
\end{itemize}
In the real case, the three following facts may be proven using geometric considerations. But elementary linear algebra leads to the same conclusion and works on any field.
\begin{fact}\label{fact:Iwazawa}
\begin{itemize}
\item All three are subgroups of $\SO(Q)$ and $H$ normalizes both $N^+$ and
$N^-$.
\item The subgroup $N^+$ is the stabilizer of $v_0=\begin{pmatrix} 1\\ 0\\
\vdots \\ 0\end{pmatrix}$ in $\SO(Q)$.
\item The group $\SO(Q)$ decomposes as the product $N^- H N^+$.
\end{itemize}
\end{fact}
With this fact, we are ready to proceed with the proof of the proposition.

\begin{proof}
First of all, we may remark that any non-zero isotropic element
$v=(x_0,\ldots,x_n)$ belongs to one of the  semi-cones. Indeed, $v$ being
isotropic, we have the equation $$x_0x_n =Q'(x_1,\ldots,x_{n-1}).$$
We assumed that $Q'$ takes only square values, hence $x_0x_n$ is either zero if
$(x_1,\ldots,x_n)=0$ or a non-zero square. In the first case, as $v\neq 0$, we
get that $x_0\neq 0$ or $x_n\neq 0$. In the second case, the class $\bar x_0$
and $\bar x_n$ are the same, as $x_0x_n\in (\bk^*)^2$. In any case, there is a
unique class $\bar \alpha$ modulo square such that $\bar\alpha \cup\{0\}$
contains both $x_0$ and $x_n$. This proves the first part of the proposition.

To prove the second point, let us remark that both $N^+$ and $N^-$ stabilize
each semi-cone. Let us justify this for $N^+$ by considering an arbitrary
element $$n=\begin{pmatrix} 1 & 2 {}^t w A & Q'(w) \\ 0 & A & w \\ 0 & 0 & 1
\end{pmatrix}.$$ Consider an element $v=\begin{pmatrix} x_0 \\ \vdots \\
x_n\end{pmatrix}$ in $\mathcal C\setminus \{0\}$. If $x_0\neq 0$ then $v$
belongs to $\mathcal C_{\overline {x_{0}}}$. And $n(v)$ has the same first
coordinate as $v$. So it also belong to $\mathcal C_{\overline {x_{0}}}$.
Otherwise, $x_0$ vanishes and so do all the $x_1,\ldots,x_{n-1}$ (as $Q(v)=0$).
In this case, $v$ belongs to $\mathcal C_{\overline {x_{n}}}$, and $n(v)$ is the
vector: $$n(v)=\begin{pmatrix} Q'(w)x_n\\ x_n w\\ x_n\end{pmatrix}.$$ It still
belongs to $\mathcal C_{\overline {x_{n}}}$.

The isomorphism $\SO(Q)/\Stab(\mathcal C_{\bar 1}) \simeq \bk^*/(\bk^*)^2$ is
now easily obtained. Indeed, using the previous fact, we may write:
$$\SO(Q)/\Stab(\mathcal C_{\bar 1})=H N^- N^+/\Stab(\mathcal C_{\bar 1}).$$ The
product $N^-N^+$ is contained in $\Stab(\mathcal C_{\bar 1})$. Hence we have a
first isomorphism:
$$\SO(Q)/\Stab(\mathcal C_{\bar 1})\simeq H/\Stab_H(\mathcal C_{\bar 1}).$$
And the stabilizer in $H$ of the semi-cone $\mathcal C_{\bar 1}$ is clearly the
subgroup: $$ \begin{pmatrix} x & 0 & 0 \\ 0 & \mathrm{Id} & 0 \\ 0&0&\frac{1}{x}
\end{pmatrix} \textrm{ for }x\in(\bk^*)^2.$$
Hence the quotient $H/\Stab_H(\mathcal C_{\bar 1})$ is isomorphic to
$\bk^*/(\bk^*)^2$
\end{proof}

\subsection{The groups $\Ad(\PSL(2,\bk))$ and $\SO(Q)$}

We focus now our attention on the case $n=2$. We note $Q(x,y,z)=xz-y^2$. In this
case, the adjoint representation is an isomorphism between $\PGL(2,\bk)$ and
$\SO(Q)$. 
The determinant modulo squares gives an isomorphism:
$$\PGL(2,\bk)/\PSL(2,\bk) \xrightarrow{\sim} \bk^*/(\bk^*)^2.$$
Under the adjoint representation, this isomorphism is exactly the same as the
one of the proposition \ref{prop:isotropic}. It translates into an isomorphism
between the quotient $\SO(Q)/\Ad(\PSL(2,\bk))$ and $\bk^*/(\bk^*)^2$. For each
class $\bar\alpha \in \bk^*/(\bk^*)^2$ represented by some $\alpha \in \bk^*$,
the following diagonal matrix belongs to the corresponding class in
$\SO(Q)/\Ad(\PSL(2,\bk))$:
$$d_\alpha = \begin{pmatrix} \alpha &  & \\& 1&\\&& \alpha^{-1}\end{pmatrix}.$$
We denote by $d_{\bar\alpha}$ the class $d_\alpha\Ad(\PSL(2,\bk))$ in
$\SO(Q)/\Ad(\PSL(2,\bk))$.
The group $\Ad(\PSL(2,\bk))$ thus identifies with the stabilizer of the semi-cones.

One may describe more precisely the case of $\bk$ a non-archimedean local field
of characteristic $\neq 2$. Recall that the group $\bk^*/(\bk^*)^2$ is of order
$4$ and isomorphic to $(\Z/2\Z)^2$ if the characteristic $p$ of the residual
field is odd. So there are $4$ semi-cones in general. For characteristic $0$ and
residual characteristic $2$, the situation is more complicated \cite{o'meara}.
Consider the case $\Q_2$: then there are $8$ classes modulo squares, and the
group is isomorphic to $(\Z/2\Z)^3$. So we have $8$ semi-cones for $\Q_2$.

We have decomposed the isotropic cone into semi-cones. We may now look at the hyperboloids, i.e. the decomposition of the other level sets of $Q$ under the action of the subgroup $\Stab(\mathcal C_{\bar 1})$. Shall we recover the hyperboloids of one or two sheets? Recall that we are looking for a model of the hyperbolic disc. In the real case, the first step is to see the two-sheeted hyperboloids. From now on, the field $\bk$ is a non-archimedean local field of characteristic different from $2$.

\subsection{Hyperboloids of one or two sheets}

Throughout this section, $\bk$ is a non-archimedean local field of
characteristic different from $2$.

The homotheties of $\bk^3$ change the value of $Q$ by a square. So, up to 
homotheties, there are $\Card(\bk^*/(\bk^*)^2)$ level surfaces for $Q$
different from the isotropic cone. Let $\bar\alpha$ be a class in
$\bk^*/(\bk^*)^2$ and $\alpha$ an element of $\bar\alpha$. Define
$$v_\alpha=\begin{pmatrix}\alpha\\ 0 \\1\end{pmatrix}.$$
Then we have $Q(v_\alpha)=\alpha$. We want to understand the stabilizer in
$\SO(Q)$ of $v_\alpha$, in order to decompose the hyperboloid $\SO(Q).v_\alpha$
into sheets. The situation is as follow:
\begin{proposition}
Let $\bar\alpha$ be a class in $\bk^*/(\bk^*)^2$ and $\alpha$ an element of
$\bar\alpha$. 
\begin{enumerate}
\item If $-1$ belongs to $\bar\alpha$, then $\SO(Q).v_\alpha$ is a one-sheeted
hyperboloid, i.e. $$\SO(Q).v_\alpha=\Ad(\PSL(2,\bk)).v_\alpha;$$
\item else $\SO(Q).v_\alpha$ is a two-sheeted hyperboloid, i.e. $\SO(Q).v_\alpha$ decomposes in two distinct $\PSL(2,\bk)$-orbits.
\end{enumerate}
\end{proposition}

\begin{proof}
\emph{First case: $-1$ belongs to $\bar\alpha$:} then the orbit
$\SO(Q).v_\alpha$ is homothetic to the orbit $\SO(Q).v_{-1}$. But we have
$Q(-1,0,1)=Q(0,1,0)$. Witt's theorem implies that the orbit $\SO(Q).v_{-1}$
coincide with the orbit 
$$\SO(Q).\begin{pmatrix}0\\1\\0\end{pmatrix}.$$
One may see that this latter orbit is a one-sheeted hyperboloid:
the group $$\Stab\begin{pmatrix}0\\1\\0\end{pmatrix}$$ contains all the matrices
$$d_\beta = \begin{pmatrix} \beta &  & \\& 1&\\&& \beta^{-1}\end{pmatrix}, \textrm{ for } \beta\in \bk^*.$$
Hence we have $$\Ad(\PSL(2,\bk)) \Stab(0,1,0)= \SO(Q),$$ 
which proves that
$$\SO(Q).(0,1,0)=\Ad(\PSL(2,\bk)).(0,1,0).$$

\emph{Second case: $-1$ does not belong to $\bar\alpha$:} then the stabilizer of
$v_\alpha$ is the orthogonal group of $Q$ restricted to $v_\alpha^{\perp}$. The
form $Q_{|v_\alpha^{\perp}}$ is equivalent to the form $Q_\alpha(u,v)=-\alpha
u^2-v^2$. The latter is anisotropic: $Q(u,v)=0$ would imply
$\alpha=-\frac{v^2}{u^2}$, so $-1$ would belong to $\bar \alpha$.
In order to understand $\Ad(\PSL(2,\bk))\Stab(v_\alpha)$, we shall understand
how the action of $\Stab(v_\alpha)$ permutes the semi-cones. 

Let $P$ be the affine plane $v_\alpha^{\perp}+\frac{1}{\alpha}v_\alpha$.
The plane $P$ is invariant under $\Stab(v_\alpha)$ and has equation:
$$\begin{pmatrix}a\\b\\c\end{pmatrix}\in P \textrm{ iff }\alpha c+a=1.$$
Look at the intersection $P\cap \C$ of $P$ and the isotropic cone. The action of
$\Stab(v_\alpha)$ on $\C$ will be transitive on the component of $P\cap \C$. So 
we compute the set of $\bar\beta \in \bk^*/(\bk^*)^2$ such that $P$
intersects $\C_{\bar\beta}$. A vector ${}^t(a,b,c)$ 
belongs to $P\cap \C$ if and only if its entries satisfy:
$$\left\{\begin{matrix} \alpha c+a & = & 1 \\ ac & = & b^2 \end{matrix}\right. .$$
We are only interested in the common class $\bar\beta$ modulo squares of $a$ and $c$ (in order to determine the semi-cone $\C_{\bar\beta}$ the solution belongs to). The $\bar\beta$'s which are solutions are exactly those verifying:
$$1\in \bar\beta+\alpha\bar\beta.$$
As $-1\not \in \bar\alpha$, this implies that $1\in \bar \beta$ or $1\in\alpha
\bar\beta$. So those $\bar \beta$ are exactly the elements of:
$$\{[\alpha +y^2] \in \bk^*/(\bk^*)^2 \textrm{ for } y\in \bk^*\}.$$ 
In other terms, this set is the norm group $\mathcal
N_{\left[\bk[\sqrt{-\alpha}];\bk\right]}$ (modulo squares) of the quadratic
extension $\bk[\sqrt{-\alpha}]$ (see \cite{o'meara}).

We know \cite[63:13a]{o'meara} that this set is always an index $2$ subgroup of $\bk^*/(\bk^*)^2$. 
As said before, the subgroup $\Stab(v_\alpha)$ permutes the $\C_\beta$'s which
intersect $P$ (by Witt's theorem). Hence it has two distinct orbits among the
$\C_{\bar\beta}$'s, and the orbit of $v_\alpha$ is a two-sheeted hyperboloid.
\end{proof}

\begin{remark}
For the very last point in the above proof, and $\bk = \Q_p$, instead of
referring to \cite{o'meara}, one may alternatively check the following without
difficulties:
\begin{itemize}
 \item $P$ always intersects the semi-cone $\C_{\bar 1}$ associated to the class
of squares,
 \item if $p$ is odd, and ${\bar\alpha}$ has an even valuation in $\bk$, $P$
intersects the two $\C_{\bar\beta}$'s for $\bar\beta$ of even valuation,
 \item if $p$ is odd, and ${\bar\alpha}$ has an odd valuation, $P$ intersects
$\C_1$ and $\C_{\bar\alpha}$,
 \item if $\bk=\Q_2$, one verifies for each class that $P$ intersects four
semi-cones. For example, in $\Q_2$, if ${\bar\alpha}$ is the class of squares,
$P$ intersects $\C_{\bar\beta}$ for $\bar\beta$ equals $\bar 1$, $\bar 2$, $\bar
5$ and $\bar{10}$.
\end{itemize}
\end{remark}

Another way to state the previous proposition is that for each subgroup $\bar K$
of index $2$ in $\bk^*/(\bk^*)^2$, there is a vector $v_\alpha$ in $\bk^3$ such
that the group $\Stab(v_\alpha)\Ad(\PSL(2,\bk))/\Ad(\PSL(2,\bk))$ is isomorphic
to $\bar K$. Those  subgroups $\bar K$ are the norm groups (modulo squares) of a
quadratic extension of $\bk$. We get the following corollary:

\begin{cor}
 Let $\bar K$ be a subgroup of index $2$ in $\bk^*/(\bk^*)^2$. 
 
 There is a unique $\bar\alpha$ in $\bk^*/(\bk^*)^2$ such that $\bar K$ is the
set $\{[\alpha+y^2] \in \bk^*/(\bk^*)^2 \textrm{ for } y\in \bk^*\}$ for any
$\alpha$ in $\bar\alpha$. The group $\bar K$ is equivalently described as the
norm group (modulo squares) of the extension $\bk[\sqrt{-\alpha}]$.

Moreover for every $\alpha$ in this $\bar\alpha$, the orbit $\SO(Q).v_\alpha$ is
a two-sheeted hyperboloid.
\end{cor}

We prefer to work with subgroups of $\bk^*$ and we hereafter denote by
$K_{\bar\alpha}$ the subgroup of $\bk^*$ such that, for any $\alpha$ in
$\bar\alpha$, we have:
 $$K_{\bar\alpha}=\{\alpha x^2+y^2 \in \bk^* \textrm{ for } x,y\in \bk^*\}$$

\section{Projectivization and duality}\label{sec:duality}

A crucial point for the (real) projective model of the hyperbolic disc consists
in the fact that the positive semi-cone over one sheet of the two-sheeted
hyperboloid is a convex cone. It allows the construction of the natural Hilbert
distance for an open convex subset of the sphere which turns out to be exactly
the hyperbolic distance.

\subsection{The positive semi-cones}

We try here to understand the ``positive semi-cone'' over one of the previously defined sheets. In other terms, we will projectivize the geometry of the previous section, but only under action of "positive" homotheties, i.e. with ratio in $K_{\bar\alpha}$. We fix an $\bar\alpha$ in $\bk^*/(\bk^*)^2$ such that $-1$ does not belong to $\bar\alpha$. We are now interested in the orbits of $K_{\bar\alpha} \Ad(\PSL(2,\bk))$ (the positive semi-cones) among the set $$\{v\in \bk^3\textrm{ such that }Q(v)\in \bar\alpha\}.$$ Once again, we will study stabilizers of points and indices of subgroups to build up the geometry of the situation.

The group $\bk^*\SO(Q)$ acts transitively on the latter set. The index of
$K_{\bar \alpha}  \Ad(\PSL(2,\bk))$ in $\bk^* \SO(Q)$ is $8$ for odd residual
characteristic ($16$ if $k=\Q_2$). Consider the usual vector 
$$v_\alpha=\begin{pmatrix}\alpha\\0\\1\end{pmatrix}$$
for some $\alpha$ in $\bar\alpha$. Of course we have $Q(v_\alpha)=\alpha\in
\bar\alpha$. And the stabilizer of $v_\alpha$ in $\bk^*\SO(Q)$ is generated by
its stabilizer in $\SO(Q)$ and the diagonal matrix $$d=\begin{pmatrix} 1 &
&\\&-1 &\\&&1\end{pmatrix}=-d_{-1}.$$ We therefore get the following lemma:
\begin{lemma}\label{lem:preduality}
Consider the set 
$$\{v\in \bk^3\textrm{ such that }Q(v)\in \bar\alpha\}.$$
The number of disjoint $K_{\bar\alpha}  \Ad(\PSL(2,\bk))$-orbits it decomposes
into is:
\begin{itemize}
 \item $4$ orbits if $-1$ belongs to $K_{\bar\alpha}$.
 \item $2$ orbits if $-1$ does not belong to $K_{\bar\alpha}$.
\end{itemize}
\end{lemma}

\begin{proof}
 We have seen that $\Ad(\PSL(2,\bk))\Stab(v_\alpha)$ is of index $2$ in
$\SO(Q)$. So the index of the group 
$$K_{\bar\alpha} 
\Ad(\PSL(2,\bk))\Stab(v_\alpha)$$ in $k^*\SO(Q)$ is $4$ or $2$ depending
on whether $d$ belongs to $K_{\bar\alpha} \Ad(\PSL(2,\bk))$ or not. Now, $-1$
belongs to $K_{\bar\alpha}$ if and only if $d$ belongs to $K_{\bar\alpha} 
\Ad(\PSL(2,\bk))$.
\end{proof}

As we are interested in the semi-cones, we will need an ad-hoc sphere, rather than the projective space:
\begin{definition}
 For $\bar\alpha$ in $\bk^*/(\bk^*)^2$ such that $-1$ does not belong to
$\bar\alpha$, the $\bar\alpha$-sphere is:
$$\S_{\bar\alpha} = (\bk^3\setminus\{0\})/K_{\bar\alpha} .$$
\end{definition}

Of course, the $\bar\alpha$-sphere is a $2$-covering of the projective space.

\subsection{Duality}\label{ssec:duality}

We are now prepared to deal with convexity properties. Convexity, in the usual
real sense, may be interpreted as a positivity condition: a subset of the plane
is convex if it is an intersection of half-spaces; or equivalently if it is the
set of points which take \emph{positive} values on a set of affine forms. We
will here follow this idea, translating "positive" into "belonging to $K_{\bar
\alpha}$".

 The polar form of $Q$ is the bilinear form $B$ defined by:
$$B(v,v')=\frac{1}{2}[Q(v+v')-Q(v-v')].$$ Recall that $\C_{\bar 1}$ is the
semi-cone associated to the class of squares:
 $$\C_{\bar 1}=\{(a,b,c)\in \bk^3\setminus\{0\} \textrm{ such that } ac=b^2\textrm{ and }a\textrm{ and }c\textrm{ are squares }\}$$

First of all, for any $\alpha$ in $\bar\alpha$, and $w \in \C_{\bar 1}$, one checks that $B(v_\alpha,w)$ belongs to $K_{\bar\alpha}$. Using the action of $K_{\bar\alpha}.\Ad(\PSL(2,\bk))$, we even get the following duality phenomenon:
$$\forall\alpha\in \bar\alpha\d{, }\forall v\in K_{\bar\alpha}\Ad(\PSL(2,\bk)).v_\alpha\d{ and }\forall w\in\C_{\bar 1}\d{, we have } B(v,w)\in K_{\bar\alpha}.$$ The hope for a possible theory of $p$-adic convexity raises up with the following theorem:
\begin{theorem}\label{the:duality}
 Let $\bar\alpha$ be an element of $\bk^*/(\bk^*)^2$ such that $-1 \not\in \bar\alpha$. 

We have equality between the two following sets:
\begin{enumerate}
 \item $\left\{K_{\bar\alpha}\Ad(\PSL(2,\bk)).v_\alpha\d{ for }\alpha\in \bar\alpha\right\}$,
 \item $\H_{\bar\alpha}=\{v\in\bk^3\d{ such that for all }w\in \C_1\d{, we have }B(v,w)\in K_{\bar\alpha}\}$.
\end{enumerate}
$\H_{\bar\alpha}$ is the union of one or two $K_{\bar\alpha}\Ad(\PSL(2,\Q_p))$-orbits: two if $-1$ belongs to $K_{\bar \alpha}$, else one.
\end{theorem}

\begin{remark}
 The projectivization of $\H_{\bar\alpha}$ will be our hyperbolic discs. One
should not be disappointed by the possibility for them to be the union of two
distinct $K_{\bar\alpha}\Ad(\PSL(2,\bk))$-orbits. This will even be our
preferred case later on. Let us recall, maybe in a yet cryptic way, that the
tree of $\PSL(2,\Q_p)$ is the union of two disjoint orbits under
$\PSL(2,\Q_p)$. 
\end{remark}

\begin{proof}
 The set $\{v\in \bk^3\textrm{ such that }Q(v)\in \bar\alpha\}$ is a single
$\bk^*  \SO(Q)$-orbit. It splits into $\H_{\bar\alpha}$ and its complementary,
which is also the image of $\H_{\bar\alpha}$ under the following matrix, for any
$y$ not belonging to $K_{\bar\alpha}$: $$d_y=\begin{pmatrix} y & &\\
&1&\\&&y^{-1}\end{pmatrix}.$$ 
Both of them are $K_{\bar\alpha}  \Ad(\PSL(2,\bk))$-invariant. Moreover, as $d_y$ normalizes $\Ad(\PSL(2,\bk))$, they decompose in the same number of orbits under $K_{\bar\alpha}  \Ad(\PSL(2,\bk))$. The previous lemma \ref{lem:preduality} implies that $\H_{\bar\alpha}$ splits into two orbits if $-1$ belongs to $K_{\bar\alpha}$, else is a single orbit. 
\end{proof}

We found our hyperbolic discs ! This section ends with the following definition:

\begin{definition}
  Let $\bar\alpha$ be an element of $\bk^*/(\bk^*)^2$ such that $-1 \not\in \bar\alpha$. 
The $\bar\alpha$-hyperbolic disc, denoted by $\D_{\bar\alpha}$, is the
projection of $\H_{\bar\alpha}$ to the $\bar\alpha$-sphere $\S_{\bar\alpha}$.
\end{definition}

Those are the main characters of our paper. We will endow them with a distance and study their geometry. Let us stress out that for each field $\bk$ one find several hyperbolic discs: one for each class $\bar \alpha$ not containing $-1$. There are $3$ of them if the residual characteristic is odd and $7$ for $\Q_2$. And they are different, in the sense that the action of $\PSL(2,\bk)$ on them is different. It may have one or two orbits and the stabilizer of a point are not conjugated for different classes $\bar\alpha$.

\section{Convexity and Hilbert distance for open sets in the projective plane of
local fields}

The present section tries to lay down the basis of a Hilbert geometry over local fields. We first define a notion of convexity inspired by (and applicable to) the hyperbolic discs just constructed ; then a natural distance for these convex sets (the Hilbert distance).

Hilbert geometries \cite{hilbert,delaharpe,beardon, vernicos} are fascinating and well-studied objects. We deliberately focus in this text on the example of the hyperbolic discs and use the convexity as a way to define an interesting structure on our $\D_{\bar\alpha}$. We nevertheless think that the existence of these examples gives a good motivation for studying in a more systematic way the hereafter proposed notion of convexity. As another example of a convex set and its Hilbert distance we describe the triangle.

\subsection{Convexity in the setting of local fields}

We propose here a definition for convexity in the setting of local fields, hoping it will prove convenient and useful. We choose to work in duality, copying the fact that a convex set in $\R^2$ may be defined as the positive side of a set of affine forms. This definition is also motivated by our example of hyperbolic discs and by the possibility (to be seen afterwards) to construct a Hilbert distance.

Fix a local field $\bk$ and $H$ a subgroup of finite index in the multiplicative group $\bk^*$. Choose an integer $n\geq 2$, $V$ a $n+1$-dimensional $\bk$-vector space and define the $H$-sphere of $V$:
$$\S_H(V)=(V\setminus\{0\})/H .$$
It is a finite covering of the projective space $\P(V)$. The $H$-spheres of $V$ and its dual $V^*$ are naturally in duality: for $a\in \S_H(V)$ and $b\in \S_H(V^*)$ the class $b(a)$ is well-defined in $\bk^*/H$. We define the ($H$-)dual $\Omega^\circ$ of a set $\Omega$ in $\S_H(V)$ as the set of forms taking values in $H$ on points of $\Omega$:
$$\Omega^\circ =\{f\in \S_H(V^*)\textrm{ such that }\forall x\in \Omega,\; f(x)\in H\}.$$

We may now take the bidual $(\Omega^\circ)^\circ$. By definition, $\Omega$ is
included in its bidual $(\Omega^\circ)^\circ$.
\begin{definition}\label{def:conv1}
A subset $\Omega$ of $\S_H(V)$ is \emph{$H$-convex} if $\Omega$ coincides with its bidual $(\Omega^\circ)^\circ$.
\end{definition}

One can alternatively say that $\Omega$ is $H$-convex if there is some set $\Omega'$ in $\S_H(V^*)$ such that $\Omega$ is the $H$-side of $\Omega'$: $\Omega$ is the set of $\omega$ such that for all $\omega'\in \Omega'$, we have $\omega'(\omega)\in H$. Indeed, one take $\Omega'=\Omega^\circ$.

The previous definition immediately leads to the definition of a convex hull:
\begin{definition}\label{def:conv2}
 The \emph{convex hull} of a subset $C$ of $\S_H(V)$ is the subset
$\Hull(C)=(C^\circ)^\circ$ of
$\S_H(V)$.
\end{definition}

Observe that $\Hull(C)$ is the smallest of $\S_H(V)$ containing $C$.

When $\bk=\R$ and $H=\R_{>0}$, definitions \ref{def:conv1} and \ref{def:conv2}
coincide with the usual definitions. When
$\bk$ is any field and $H=\bk^*$, the projection to
the $H$-sphere of the complement to a finite union of hyperplanes is
an example of convex set. And, in view of our theorem \ref{the:duality}, the
hyperbolic discs $\D_{\bar\alpha}$ we just defined are $K_{\bar\alpha}$-convex.

Any polytope will be $H$-convex: take a finite set of forms and their common $H$-side. One may describe more precisely the triangle: for any local field $\bk$ and a finite index subgroup $H$ of $\bk^*$, we define the $H$-triangle $\mathbf T_H$: it is the set of points $[x,y,z]$ in the $H$-sphere whose three coordinates are in $H$. Then $\mathbf T_H$ is the dual of set $\{e_1,e_2,e_3\}$ of the three coordinates forms.

\subsection{Hilbert distance, revisited}

This subsection is devoted to a rephrasing of the Hilbert distance in the real case. We want to redefine it without any mention to the ordering on $\R$. This is possible, even if the definition proposed might seem highly artificial for this real case. We will then move on in the next subsection to other fields, trying to transpose our new definition.

\begin{figure}[ht]
\begin{center} 
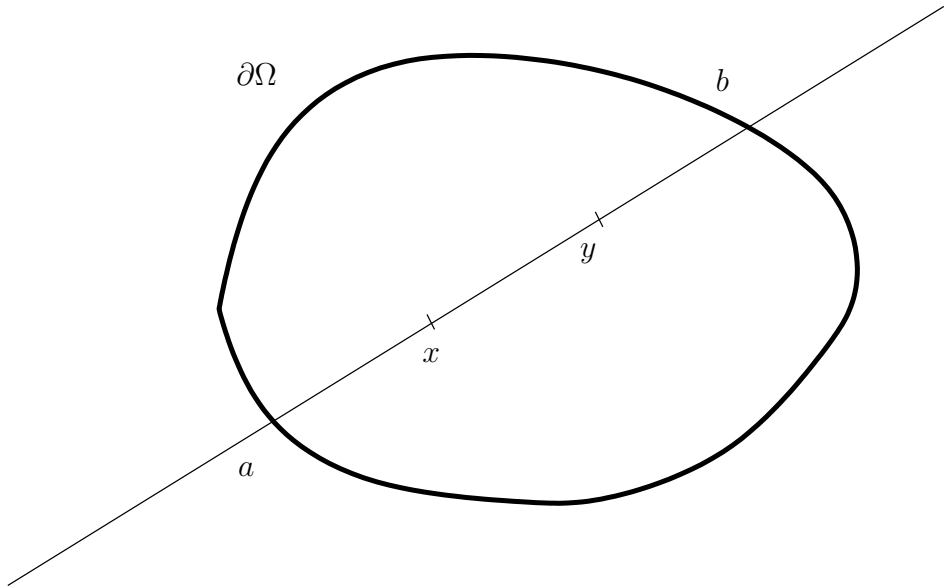
\end{center}
\caption{The Hilbert distance on $\Omega$}\label{fig:distanceHilbert}
\end{figure}

We fix here an open, relatively compact and convex set $\Omega$ in the space $\R^n$ and take $V=\R^{n+1}$, $H=\R_{>0}$. Via the adjunction of an hyperplane at infinity to $\R^n$, $\Omega$ becomes an open and proper convex set in the projective space $\P(V)$. So we have a convex lift (still called $\Omega$) in the (usual) sphere.
The Hilbert distance is classically defined in the following way:

for $x$ and $y$ in $\Omega$, let $a$ and $b$ be the intersections between the line $(xy)$ and the frontier $\partial \Omega$ of $\Omega$, such that $a$, $x$, $y$, $b$ are in this order on the line $(xy)$ (see figure \ref{fig:distanceHilbert}). Consider (noting $zt$ the distance between two points $z$ and $t$):
$$D_\Omega(x,y)=[a,b,x,y]=\frac{ay}{ax}\frac{bx}{by}$$
and take its logarithm:
$$d_\Omega(x,y)=\ln\left([a,b,x,y]\right)=\ln\left(\frac{ay}{ax}\frac{bx}{by}\right).$$
It is well known that $d_\Omega$ is a distance \cite{hilbert,beardon}. The only point necessitating a proof is the triangular inequality. Moreover, it is invariant under projective transformation. Once again, we will not get further into the  theory of Hilbert distance. We just want to define it in another way.

The first problem of this definition for other fields than $\R$ is the word "boundary".  In totally disconnected fields such boundaries tend to be void. We prefer to use the duality. So if $\phi$ and $\phi'$
belong to $V^*$, $x$ and $x'$ belong to $V$ with neither $\phi(x)$ nor $\phi'(x')$ null, we note:
$$[\phi,\phi',x,x']=\frac{\phi(x')}{\phi(x)}\frac{\phi'(x)}{\phi'(x')}$$

This formula is invariant under homothety on $V$ or $V^*$, and under the action of $\GL(V)$ conjointly on $V$ and $V^*$.
\begin{lemma}
 Let $\Omega$ be an open proper convex set in the sphere, $\Omega^\circ$ be its dual. Then we have, for all $x$ and $y$ in $\Omega$:

$$D_\Omega(x,y)=\max_{\phi,\, \phi'\in \Omega^\circ} [\phi,\phi',x,y]$$
\end{lemma}

So the Hilbert distance is the logarithm of this maximum.

\begin{proof}
 Consider $a'$ (resp. $b'$) the intersection point between $\ker(\phi)$ (resp.
$\ker(\phi')$) and the line $(xy)$. By convexity, and the fact that
$\phi\in\Omega^\circ$, $a'$ and $b'$ do not belong to $\Omega$. And the theorem
of Thales implies that $\frac{\phi(x)}{\phi(y)}$ equals $\frac{a'-x}{a'-y}$ (see
figure \ref{fig:Thales}) and the same with $\phi'$ and $b'$. Hence we get the
equality between $[\phi,\phi',x,y]$ and $[a',b',x,y]$. The maximum of the latter
is attained for $a'=a$ and $b'=b$, i.e. $\ker(\phi)$ a supporting hyperplane of
$\Omega$ through $a$, and $\ker(\phi')$ through $b$.
\end{proof}

\begin{figure}[ht]
 \begin{center}
 \scalebox{.7}{ 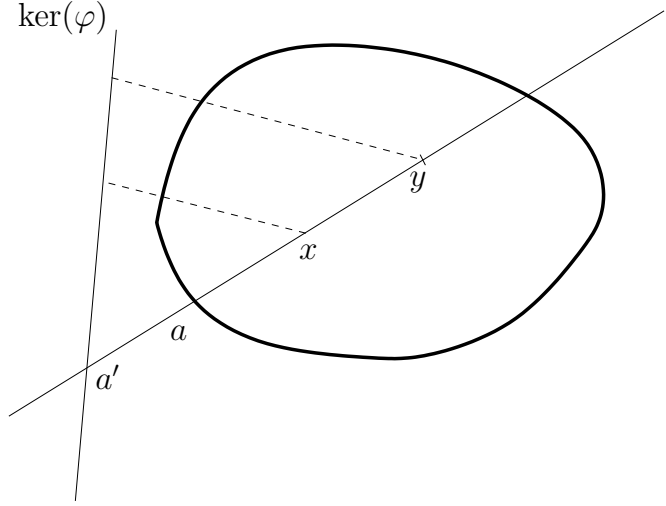}
 \end{center}
\caption{Rewriting the Hilbert distance}\label{fig:Thales}
\end{figure}

At this point we do not mention $\partial \Omega$ any more, which is the first step. But we use the notion of maximum, unavailable in other fields. The following step is given by this lemma:
\begin{lemma}\label{lem:distanceintervalle}
 Let $\Omega$ be an open proper convex set in the sphere, $\Omega^\circ$ be its dual. Then we have, for all $x$ and $y$ in $\Omega$:

$$\{ [\phi,\phi',x,y]\textrm{ for }\phi,\, \phi'\in \Omega^\circ \} =[ D_\Omega(x,y)^{-1},D_\Omega(x,y)].$$
\end{lemma}

\begin{proof}
 First of all, we check that $[\phi,\phi',x,y]=[\phi',\phi,x,y]^{-1}$, so the first set contains $D_\Omega(x,y)$ (its maximum by the previous lemma) and $D_\Omega(x,y)^{-1}$ and is contained in the interval $[ D_\Omega(x,y)^{-1},D_\Omega(x,y)]$. Now $\Omega^\circ$ is convex hence connected. We conclude by continuity.
\end{proof}

We may now redefine the Hilbert distance via the Haar measure on $\R^*$. Choose the Haar measure $\mu$ on $\R^*$ defined by the following, for $t>1$:
$$\mu[t^{-1},t]=\ln(t).$$
Then, the previous lemma yields immediately:
\begin{proposition}
 Let $\Omega$ be an open proper convex set in the sphere, $\Omega^\circ$ be its dual. Then the distance $d_\Omega(x,y)$ is given by:

$$d_\Omega(x,y)=\mu\left( \{ [\phi,\phi',x,y]\textrm{ for }\phi,\, \phi'\in \Omega^\circ \} \right).$$
\end{proposition}

We have reached our goal: the latter number may be defined on any field and recovers the Hilbert distance in the real case. 

\begin{remark}
 In this form, the triangular inequality becomes easy, once we check that (with obvious notation) $[\phi,\phi',x,y]=[\phi,\phi',x,z][\phi,\phi',z,y]$ (see lemma \ref{lem:triangineq}). Easy indeed, but one still needs to use some properties of the real numbers, e.g. an ordering. This will be a further problem.
\end{remark}

\subsection{A generalized Hilbert distance}\label{sec:genHilbdist}

We now come back to a more general situation: let $\bk$ be $\R$ or some non-archimedean local fields $\bk$ (of characteristic different from $2$) and $|.|$ its norm. Let $H$ be a finite index subgroup of $\bk^*$ and $V$ an $n+1$-dimensional $\bk$-vector space. Fix a Haar measure $\mu$ on $\bk^*$. We define a notion of symmetric ball:
\begin{definition}
 The symmetric ball of radius $r$ is the set $$\{x\in \bk\textrm{ such that } |x-1|\leq r \textrm{ and } |x^{-1}-1|\leq r\}.$$
\end{definition}

We may also define the notion of proper convex set, extending the notion of properness in the real case:
\begin{definition}
 Let $\Omega$ be a  $H$-convex set in $S_H(V)$. The set $\Omega$ is proper if the intersection $\displaystyle \cap_{\phi \in \Omega^\circ} \ker(\phi)$ equals $\{0\}$.
\end{definition}

We are now able to define a Hilbert distance on proper convex sets.
\begin{theorem}\label{the:distance}
 Let $\Omega$ be an open and proper $H$-convex set in $S_H(V)$, $\Omega^\circ$ be its dual. For $x$, $y$ in $\Omega$, define $d_\Omega (x,y)$ as the measure for $\mu$ of the smallest symmetric ball containing $\left\{ [\phi,\phi',x,y]\textrm{ for }\phi,\, \phi'\in \Omega^\circ \right\}$.

Then $d_\Omega$ is a distance on $\Omega$.
\end{theorem}

\begin{proof}
For this proof, we note $B_\Omega(x,y)$ the smallest symmetric ball containing the set $\{ [\phi,\phi',x,y]\textrm{ for }\phi,\, \phi'\in \Omega^\circ \}$.

First of all, if $x\in \Omega$, we have $d_\Omega(x,x)=\mu(\{1\})=0$. Moreover,
if $x\neq y$ are in $\Omega$, fix $X$ and $Y$ some representatives in $V$
(recall that $\Omega$ lives in the $H$-sphere). We cannot have some element
$l\in k$ such that $\phi(X)=l\phi(Y)$ for all $\phi\in \Omega^\circ$, because
$X-lY$ would be in every $\ker(\phi)$, contradicting the properness. Hence the
set $\left\{ [\phi,\phi',x,y]\textrm{ for }\phi,\, \phi'\in \Omega^\circ
\right\}$ is not restricted to $\{1\}$ and the smallest symmetric ball
containing it has a non empty interior. Its measure is not $0$ and $d_\Omega
(x,y)\neq 0$.

We have $[\phi,\phi',x,y]=[\phi',\phi,y,x]$. Hence we have $B_\Omega(x,y)=B_\Omega(y,x)$ and $d_\Omega$ is symmetric: $d_\Omega(x,y)=d_\Omega(y,x)$.

We have already mentioned that
$[\phi,\phi',x,y]=[\phi,\phi',x,z][\phi,\phi',z,y]$. Hence, $B_\Omega(x,y)$ is
included in the set $B_\Omega(x,z).B_\Omega(z,y)$. 
We check the following lemma:
\begin{lemma}\label{lem:triangineq}
 If $B$ and $B'$ are two symmetric balls, we have $\mu(BB')\leq \mu(B)+\mu(B')$.
\end{lemma}

\begin{proof}
If $\bk=\R$, the symmetric balls of radius $t$ in $\R_+^*$ is $[t^{-1},t]$ and
we have $\mu([t^{-1},t])=\ln(t)$ (up to a constant). Hence we have
$$\mu\left(\left[t^{-1},t\right]\left[s^{-1},s\right]\right)=\mu\left(\left[t^{
-1}s^{-1},st\right]\right)= \ln(t)+\ln(s).$$

If $\bk$ is a non archimedean local field, note $\mathcal O$ the ring of integers, $q^{-1}$ the norm of an uniformizer (for $\bk=\Q_p$, this means $\mathcal O= \Z_p$ and $q=p$), let $B_t$ be the symmetric ball of radius $q^t$, $B_s$ of radius $q^s$ (with $t\leq s$ integers). Then $\mu(B_t)=\frac{1}{q-1}q^{t+1}$ if $t< 0$ and $t+1$ if $t\geq 0$ (normalizing $\mu$ by $\mu(\mathcal O^*)=1$).
Moreover one checks that the product $B_tB_s$ is $B_s$ if $t\leq 0$ or $B_{t+s}$. Hence the results also hold in this case. 
\end{proof}
It yields that $d_\Omega$ verifies the triangular inequality, and even in the $p$-adic case, an ultrametric inequality if a distance is lesser than $1$.
\end{proof}

Remark that in the real case, we just redefined the Hilbert distance, nothing
more. We hope that this definition may give some nice non-standard Hilbert
geometries. We shall try to get some insights on the possible geometries elsewhere. We focus in this paper on the first important examples,
namely the hyperbolic discs $\D_{\bar\alpha}$. We study in the following section
their geometry.

Let us discuss a bit the triangle before that. Let $\bk=\Q_p$ and $H$ be the
subgroup of squares. Then one checks that the dual $(\mathbf T_H)^\circ$ of the
triangle is composed of exactly three points in the $H$-sphere of $(\Q_p^3)'$:
the projection of the three coordinate forms, denoted $e_1$, $e_2$ and $e_3$.

As the dual is finite, we really need to consider symmetric balls to fill up the sets of cross-ratios involved in the definition of distance. Indeed, if we had not filled up, we would always compute the measure of a finite set. With the definition we gave, take two points $P_1=[p^{n_1}x:1:p^{n_2}y]$ and $P_2=[p^{m_1}a:1:p^{m_2}b]$ in the triangle (with $x$, $y$, $a$ and $b$ in $\Z_p^*$). Let $N=n_1-m_1$, $M=n_2-m_2$. Then we have (normalizing the Haar measure $\mu$ on $\Q_p$ such that $\mu(\Z_p^*)=1$:
\begin{proposition}
The distance in $\mathbf T_H$ between $P_1$ and $P_2$ is the following:
 \begin{itemize}
 \item If $N=M=0$, then the distance is less than $1$.
 \item Else, it is $\d{max}\left\{N,M,-N,-M,N-M,M-N\right\}+1$.
 \end{itemize}
 Moreover, on the balls of radius $1$, the distance is ultrametric.
\end{proposition}

The proof is a direct application of the definition. An interesting consequence
is the following one: with the notation above, one can define a map $\pi$ from
$\mathbf T_H$ to $\Z[j=e^{\frac{2i\pi}{3}}]$ by sending $[p^{n_1}x:1:p^{n_2}y]$
to $(n_1+n_2j)$. This map shrinks the balls $B_{n_1,n_2}$ of radius $1$ in
$\mathbf T_H$ to a point in $\Z[j]$. But if you equip $\Z[j]$ with the hexagonal
norm, the map sends two points of $\mathbf T_H$ at distance $d$ to two points in
$\Z[j]$ at distance $d-1$. The figure \ref{fig:triangle} shows the image of the
ball of center $[1:1:1]$ and of radius $2$. This exhibits a striking analogy
with the Hilbert distance on the triangle in the real case, which is isometric
to the hexagonal norm on the plane $\R^2$.

\begin{figure}[ht]
\begin{center} 
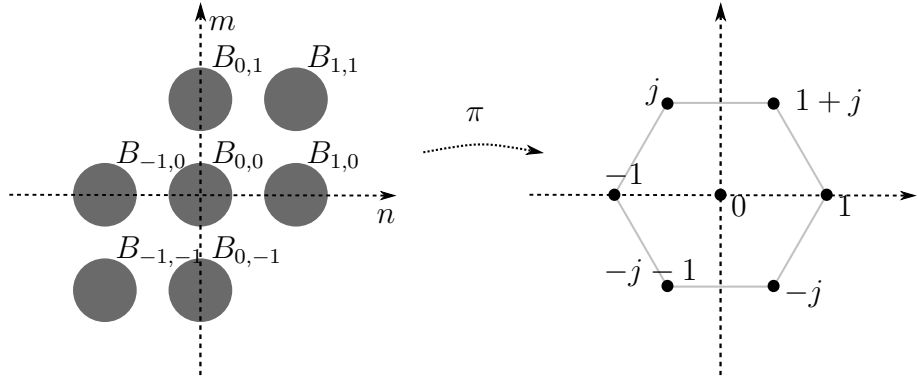
\end{center}
\caption{The Hilbert distance on the triangle}\label{fig:triangle}
\end{figure}

\section{Geometry of the hyperbolic discs}

The hyperbolic discs give nice examples of convex sets. The generalized Hilbert distance is defined and endow them with a geometry.  We will describe a bit this geometry: what are their duals and isometry groups. Then we will actually compute the Hilbert distance. 

Recall the setting of sections \ref{sec:hyperboloid} and \ref{sec:duality} : we
work in a non-archimedean local field $\bk$ of characteristic $\neq 2$,
${\bar\alpha}$ is a fixed class in $\bk^*/(\bk^*)^2$ which does not contain
$-1$. We associated to it a subgroup $K_{\bar\alpha}$ of index $2$ in
$\bk^*/(\bk^*)^2$. We studied the quadratic form $Q(x,y,z)=xz-y^2$, and called
$B$ its polar form. We have defined the isotropic semi-cone $\C_{\bar 1} \subset
\bk^3$ and we denoted $\C_{\bar\alpha}$ its projection in the
$K_{\bar\alpha}$-sphere\footnote{It is only a slight abuse of notation, as the
projections of the semi-cones $\C_{\bar 1}$ and $\C_{\bar\alpha}$ in the
$K_{\bar\alpha}$-sphere are the same.}. Thanks to $B$ we have an identification
between $\bk^3$ and its dual, and so between their $K_{\bar\alpha}$-spheres.
With the definitions of the previous section, the meaning of theorem
\ref{the:duality} is that the disc $\D_{\bar\alpha}$ is the convex
$(\C_{\bar\alpha})^\circ$.

\subsection{The duals}\label{sec:dualisom}

The description of the duals of the discs is required to compute the generalized
Hilbert distance.

\begin{proposition}
If $-1$ is a square in $\bk$, or if ${\bar\alpha}\neq 1$, then the dual $\D_{\bar\alpha}^\circ$ of $\D_{\bar\alpha}$ is exactly the semi-cone $\C_{\bar\alpha}$.

If $-1$ is not a square, and ${\bar\alpha}=1$, then the dual $\D_{\bar 1}^\circ$ is $\D_{\bar 1}\cup\C_{\bar 1}$.
\end{proposition}

\begin{proof}
 We know by theorem \ref{the:duality} that $\C_{\bar\alpha}$ is included in $\D_{\bar\alpha}^\circ$ and even that $\D_{\bar\alpha}=\C_{\bar\alpha}^\circ$. Moreover $\D_{\bar\alpha}^\circ$ does not intersect $\C_{\bar\beta}$ for $\beta\not\in K_{\bar\alpha}$: $$B((\alpha,0,1),(\beta,0,0))=\beta\not\in K_{\bar\alpha}.$$
 Now, take some $v$ outside  of the isotropic cone. If $v$ does not belong to
$\D_{\bar\alpha}$ we may use the action of $\PSL(2,\bk)$ to send it (at the
limit) in a semi-cone $\C_{\bar\beta}$ for $\bar\beta \neq {\bar\alpha}$ (the
isotropic cone is the limit set for the action of $\PSL(2,\bk)$ on
$\mathbb P(\bk^3)$). Hence $v$ does not belong to
$\D_{\bar\alpha}^\circ$.

If $v$ belongs to $\D_{\bar\alpha}$, up to the action of $\PSL(2,\bk)$ (and homothety), one may assume that $v=(\alpha x^2,0,1)$ for some $x\in \bk$. We then get that, for any $v'=(a,b,c)\in \D_{\bar\alpha}$:
$$B(v,v')=(\alpha x^2 c+a) \in K_{\bar\alpha} $$
Choose $v'=(\alpha y^2,0,1)$. Then, we get $B(v,v')=\alpha (x^2+y^2)$. This belongs to $K_{\bar\alpha}$ for any $y$ if and only if $-1$ is not a square, and ${\bar\alpha}=\bar 1$.

We conclude by the following: if $-1$ is not a square, and $\bar\alpha= \bar{1}$, take, $v$ and $w$ in $\D_{\bar{1}}$. Then up to the action of $K_{\bar 1}\Ad(\PSL(2,\bk))$, one may assume that $v=(x^2,0,1)$ and $w=(y^2,0,1)$ for some $x$, $y$ in $\bk$. And we have $B(v,w)=x^2+y^2\in K_{\bar{1}}$.
\end{proof}

The second situation described ($-1$ not a square, $\bar \alpha=\bar 1$) reminds us of the real case. The existence of the other ones shows the limits of the analogy. But in any case, the dual is big enough and $\D_{\bar\alpha}$ is a proper and open $K_{\bar\alpha}$-convex set:
\begin{cor}
 The hyperbolic disc $\D_{\bar\alpha}$ is a  proper and open $K_{\bar\alpha}$-convex set. It inherits a Hilbert distance $d_{\D_{\bar\alpha}}$.
\end{cor}

We will effectively compute the distance in section \ref{sec:compdist}. Before that, let us try to describe a bit the geometry and isometries of these discs.

\subsection{Lines, short and long}

We have a natural notion of line in the disc $\D_{\bar\alpha}$. Consider a point $v\in \D_{\bar\alpha}$. Any (linear) plane in $\bk^3$ projects to a projective line in $\S_{\bar\alpha}$. We call \emph{line} in $\D_{\bar\alpha}$ the intersection of a line in $\S_{\bar\alpha}$ with $\D_{\bar\alpha}$.

Hence through two distinct points in $\D_{\bar\alpha}$, there is a unique line in $\D_{\bar\alpha}$. But there are two distinct kind of lines: short lines, which are compact in $\S_{\bar\alpha}$, and long lines, for which the closure intersects $\C_{\bar\alpha}$ in exactly two points.

Indeed, consider a point $v\in \D_{\bar\alpha}$. A line through $v$ is determined by a vector $w$ in $v^\perp$, such that the line is the projectivization of the plane $P_w$ generated by $v$ and $w$. Now we have seen that $Q$ restricted to $v^\perp$ is equivalent to the form $-\alpha x^2 -y^2$. Hence it takes values exactly in $-K_{\bar\alpha}$. So there are two cases:
\begin{itemize}
 \item If $Q(w)\in -\bar\alpha$, then $Q$ restricted to $P_w$ is isotropic and the projection of $P_w$ in $\S_{\bar\alpha}$ intersects $C_{\bar\alpha}$ in two points. We call it a \emph{long line}.
 \item If $Q(w)\not\in -\bar\alpha $, then $Q$ restricted to $P_w$ is anisotropic and the projection of $P_w$ in $\S_{\bar\alpha}$ does not intersect $C_{\bar\alpha}$. Then the projection of $P_w$ in $\S_{\bar\alpha}$ is the union of two disjoint compact sets: the points in $\D_{\bar\alpha}$ and its complementary. We call it a \emph{short line}.
\end{itemize}

For a line $l$ in $\D_{\bar\alpha}$, we note $P(l)$ the plane in $\bk^3$ such that $l$ is the projectivization of $P(l)$ (intersected with $\D_{\bar\alpha}$).
For two lines through a point $v$, there is a well-defined notion of orthogonality, thanks to the bilinear form $B$:
\begin{definition}
 Two lines $l_1$ and $l_2$ through a point $v$ in $\D_{\bar\alpha}$ are orthogonal if $P(l_1)\cap v^\perp$ and  $P(l_2)\cap v^\perp$ are orthogonal for $B$.
\end{definition}

One checks that if $-1$ is not a square and $\bar \alpha=\bar 1$, then the orthogonal of a long line is a long line.

\subsection{The projective isometry group}

We are now able to describe the group of projective isometries, and the
transitivity of its action. We will discuss later (see section \ref{sec:carac})
the existence of non-projective isometries.

\begin{proposition}\label{prop:transit}
The group $\Isom(\D_{\bar\alpha})$ of projective maps in $\GL(3,\bk)/K_{\bar\alpha}$ preserving $\D_{\bar\alpha}$ acts by isometries on $\D_{\bar\alpha}$. It is isomorphic to $\PGL(2,\bk)$, an element $g$ of $\PGL(2,\bk)$ acting by $\det(g)\Ad(g)$.

Its action is transitive on $\D_{\bar\alpha}\times \C_{\bar\alpha}$.

Its action is transitive on the sets of long lines, and of short lines.

Its action is transitive on the flags ``a point in a long line''.

Its action preserves orthogonality.
\end{proposition}

\begin{remark}
 From now on, the action of $\PGL(2,\bk)$ on $\D_{\bar \alpha}$ will always be the one described above.
\end{remark}

\begin{proof}
The first part is classical in Hilbert geometry: a projective transformation preserves $\D_{\bar\alpha}$ if and only if it preserves its dual. Hence it preserves the Hilbert distance defined. Moreover, an element of $\GL(3,\bk)/K_{\bar\alpha}$ which preserves $\C_{\bar\alpha}$ preserves the isotropic cone of $Q$. So it belongs to the projective orthogonal group $\PO(Q)$. As it shall preserve $\D_{\bar\alpha}$, one easily sees that it is the group described.

Now the stabilizer of $\C_{\bar\alpha}$ is generated by $\Ad(\PSL(2,\bk))$ and $d_\alpha$.
 Fix the point $K_{\bar\alpha}.(1,0,0)$ in $\C_{\bar\alpha}$. Its stabilizer is generated by $d_\alpha$ and $\Ad(P)$ where $P$ is the parabolic subgroup $\begin{pmatrix} x & y \\ 0 &x^{-1}\end{pmatrix}$. Moreover any point in $\C_{\bar\alpha}$, represented by a triple $(x^2,xy,y^2)$ is the image of $(1,0,0)$ under $\Ad\begin{pmatrix} x & 0\\ \frac{y}{2} & x^{-1} \end{pmatrix}$, proving the transitivity of the action on $\C_{\bar\alpha}$.
Remark that any point in $\D_{\bar\alpha}$ is represented by a triple $(\alpha c^{-2}+c^{-2}b^2,cb,c^2)$ (up to the action of $K_{\bar\alpha}$). By the action of $\Ad(P)$ (namely of $\Ad\begin{pmatrix} c & b\\ 0 & c^{-1} \end{pmatrix}$), this point is sent to $v_\alpha=(\alpha,0,1)$. This proves the first transitivity claimed.

Now fix a short line $l$ (resp a long line $L$) through $v_\alpha$. Take another short  line $l'$ (resp. long line $L'$). Using the transitivity of $\Isom(\D_{\bar\alpha})$ on $\D_{\bar\alpha}$, we send $l'$ (resp $L'$) on a short (resp. long) line through $v_\alpha$. Eventually the stabilizer of $v$ acts transitively on the set of directions $<w>$ in $v^\perp$ such that $Q(w) \not \in -\bar\alpha$ (resp. $Q(w) \in -\bar\alpha$). So you may send $l'$ to $l$, and $L'$ to $L$.

The stabilizer of the long line $\{(x,0,y),x,y\in K_{\bar\alpha}\}$ acts
transitively on this line (indeed, it contains all the diagonal matrices with
entries in $K_{\bar\alpha}$).  Using the transitivity on the long lines, we get
the transitivity on the flags.

The last point is straightforward.
\end{proof}

We may describe more precisely the action of $\PGL(2,\bk)$ on $\C_{\bar \alpha}$:
\begin{fact}
 The action of $\PGL(2,\bk)$ on $\C_{\bar \alpha}$ is isomorphic to its projective action on $\P_1(\bk)$ via the bijection:
$$K_{\bar\alpha}(x^2,xy,y^2) \mapsto \bk(x,y).$$
\end{fact}


\subsection{The Hilbert distance on the discs}\label{sec:compdist}

Of course it is possible to effectively compute the distance. We only prove the result when $\bk=\Q_p$ with $p\neq 2$ and $\bar\alpha$ has even valuation. 
Here we choose the Haar measure on $\Q_p$ such that $\mu(\Z_p^*)=1$ and fix $\alpha \in \bar\alpha$ of valuation $0$. Thanks to the transitivity of the isometry group, we just have to compute two distances: first the distance between the point $v_\alpha=(\alpha,0,1)$ and a point $v_1=(\alpha x^2,0,1)$ (for which $v_\alpha$ and $v_1$ define a long line), second the distance between $v_\alpha$ and some $v_2= ((1-ay)\alpha,1,1+ay)$, where $a\in \Z_p^*$ is such that $1+\alpha a^2$ does not belong to $-\bar\alpha$. Indeed in the second case, one checks that $v_\alpha$ and $v_2$ define a short line.
\begin{proposition}\label{prop:calcdist} Assume $p\neq 2$ and $\bar\alpha$ has even valuation.

\emph{First case :} The Hilbert distance $d_{\D_{\bar\alpha}}$ between the points $v_\alpha=(\alpha,0,1)$ and $v_1=(\alpha x^2 , 0,1)$ on a long line is given by
\begin{itemize}
\item $d_{\D_{\bar\alpha}}(v_\alpha,v_1)=2n+1$, if $\max(|x^2-1|,|x^{-2}-1|)=p^{2n}\geq 1$.
\item $d_{\D_{\bar\alpha}}(v_\alpha,v_1)=\frac{1}{p-1}p^{n+1}$, if $\max(|x^2-1|,|x^{-2}-1|)=p^n<1$.
\end{itemize}

\emph{Second case :} The Hilbert distance $d_{\D_{\bar\alpha}}$ between the points $v_\alpha=(\alpha,0,1)$ and $v_2=((1-ay)\alpha,y,(1+ay))$ on a short line is always lesser than one. If $|y|=1$, it is $1$. If $|y|=p^n<1$, it is given by $\frac{1}{p-1}p^{n+1}$.
\end{proposition}

\begin{remark}
 If $\bar \alpha $ has an odd valuation, one checks that the only modifications needed are the following. In the first case (long line), if $\max(|x^2-1|,|x^{-2}-1|)=p^{2n}\geq 1$, then one gets $d_{\D_{\bar\alpha}}(v_\alpha,v_1)=2n+\frac{1}{2}$. In the second case (short line), if $|y|=1$, the distance becomes $\frac{1}{2}$.
\end{remark}

\begin{proof}
We sketch the computation.

 \emph{First case:} Fix $w=(a,b,c)$ and $w'=(a',b',c')$ in the dual $\D_{\bar\alpha}^\circ$, choosing $a$, $a'$, $c$ and $c'$ squares in $\Z_p$. We want to evaluate the cross-ratio:
$$\frac{B(w,v_1)}{B(w,v_\alpha)}\frac{B(w',v_\alpha)}{B(w,v_1)}=\frac{\alpha c x^2+a}{\alpha c +a}\frac{\alpha c' +a'}{\alpha c' x^2 +a'}.$$
Taking $w=(1,0,0)$ and $w'=(0,0,1)$, the cross-ratio takes the value $x^2$. Permuting $w$ and $w'$, it equals $x^{-2}$. Let us show that all the cross-ratios belong to the smallest symmetric ball containing $x^2$.

Suppose first that $|x^2|=p^{2n}\geq 1$. It is easily seen that the first ratio has a norm between $1$ and  $p^{2n}$ (recall that there is no simplification between a square and $\alpha$ times a square). The second ratio has a norm between $p^{-2n}$ and $1$. Hence the cross-ratio has a norm between  $p^{-2n}$ and $p^{2n}$. This means it belongs to the symmetric ball in $K_{\bar\alpha}$ containing $x^2$ and $x^{-2}$. And this symmetric ball is the union: $$\cup_{-n \leq k \leq n} p^{2k}\Z_p^*.$$ It has the stated measure.

If we have $|x^2-1|=p^{n}<1$, i.e. $x^2=1+x'$, then we may rewrite the cross-ratio:
$$\frac{B(w,v_1)}{B(w,v_\alpha)}\frac{B(w',v_\alpha)}{B(w,v_1)}=\frac{1+x'\frac{\alpha c}{\alpha c +a}}{1+x'\frac{\alpha c'}{\alpha c' +a'}}.$$ It is closer to $1$ than $x^2$, so belongs to the smallest symmetric ball containing $x^2$, i.e. to $1+p^{n} \Z_p$. It has measure $\frac{1}{p-1}p^{n+1}$.

\emph{Second case:} first of all, we have $Q(v_2)=\alpha - (1+\alpha a^2) y$. As it belongs to $\bar\alpha$ and we supposed that $-(1+\alpha a^2)$ does not belong to $\bar\alpha$, it implies that $|y|\leq 1$. Fix $w=(d,e,f)$ and $w'=(d',e',f')$ in $\D_{\bar\alpha}^\circ$. We compute the cross-ratio:
$$\frac{B(w,v_2)}{B(w,v_\alpha)}\frac{B(w',v_\alpha)}{B(w,v_2)}=\frac{1-ay(1+\frac{2e}{a(\alpha f+d)})}{1-ay(1+\frac{2e'}{a(\alpha f'+d')})}.$$
One sees that the smallest symmetric ball containing these cross-ratios is $\Z_p^*$ if $|y|=1$, else $1+|y|\Z_p$.
\end{proof}

A remark on this proposition: the values taken by $d_{\D_{\bar\alpha}}$ are $\frac{1}{p-1}p^{n+1}$ for negative $n$, and $2n+1$ for positive $n$ (or $2n+\frac{1}{2}$). It proves that this distance is ultrametric at distance less than one. We will see that at large scales, it is not any more ultrametric. 
We define the ultrametric locus around a point:
\begin{definition}
For some point $\omega$ in $\D_{\bar\alpha}$, we denote by $\mathcal U(\omega)$
its ultrametric neighbourhood: $$\mathcal U(\omega)=\{\omega'\textrm{ such that
}d_{\D_{\bar\alpha}}(\omega,\omega')\leq 1\}.$$ 
\end{definition}

Remarky that, for two points in a long line, we have a much more precise notion than the distance:
\begin{definition}\label{def:multdist}
 Let $v$ and $v'$ be two points in $\D_{\bar\alpha}$ defining a long line $L$. Let $w$ and $w'$ be the two intersection points between $L$ and the isotropic semi-cone, $\phi=B(w,.)$ and $\phi'=B(w',.)$.

We denote by $D_{\bar\alpha}(v,v')$ the set $\left\{ [\phi,\phi',v,v'] ,\,
[\phi',\phi,v,v']\right\}$, and call it the multiplicative (Hilbert) distance
between $v$ and $v'$. 
\end{definition}

The Hilbert distance $d_{\bar\alpha}$ is a function of $D_{\bar\alpha}$, justifying the name of multiplicative distance:
\begin{fact}
 Let $v$ and $v'$ be two points in $\D_{\bar\alpha}$ defining a long line. Let $\{x,x^{-1}\}=D_{\bar\alpha}(v,v')$. Let $p^n=\max\{|x-1|,|x^{-1}-1|\}$.

Then $d_{\bar\alpha}$ is given by $n+1$ if $n\geq 0$, else by $\frac{1}{p-1}p^{n+1}$.
\end{fact}

\section{Links with the tree}\label{sec:tree}

We clarify here the links between the hyperbolic disc associated to an ${\bar\alpha}$ of even valuation in $\Q_p$ with $p\neq 2$ and the more classical tree $\T$ of $\PSL(2,\Q_p)$. We will not treat the case of a general non-archimedean local field in order to avoid heavy notations. However, it should be clear that the same phenomenon occurs in this more general case. The situation is a bit different for the discs associated to an ${\bar\alpha}$ of odd valuation. Let us recall briefly that the tree may be defined as follows \cite{serre}:
\begin{itemize}
 \item The vertices are the orders (up to isometry) in $\Q_p^2$, i.e. the free $\Z_p$-modules of rank $2$. It is also the set $\PGL(2,\Q_p)/\PGL(2,\Z_p)$.
 \item Two orders are linked by an edge if they are of index $p$ one in the other.
\end{itemize}
The action of $\PSL(2,\Q_p)$ has two distinct orbits: the orbit of $\Z_p^2$ and
the orbit of $\Z_p\cdot (1,0)\oplus\Z_p\cdot (0,p)$. Two linked vertices belongs
to different orbits. One sees that this graph is a complete
$p+1$-tree\footnote{Let us also mention that one may interpret the whole tree
(edges included) as the set of norms on $\Q_p^2$ \cite{goldmaniwahori}.}.
Its boundary at infinity is naturally identified with $\P_1(\Q_p)$. Two distinct
points in $\P_1(\Q_p)$ represented by vectors $v_1$ and $v_2$ define a
unique geodesic in the tree, composed by the orders of the form $\Z_p \cdot
xv_1\oplus \Z_p\cdot yv_2$.

The hyperbolic disc $\D_{\bar\alpha}$ and the tree are both homogeneous sets
under $\PGL(2,\Q_p)$. The normalizer of a point in $\D_{\bar\alpha}$ is a
compact subgroup, so it is included in a conjugate of $\PGL(2,\Z_p)$. We will
check (by a tedious computation, unfortunately) that it is included in only one
maximal compact subgroup. This holds because we supposed ${\bar\alpha}$ has even
valuation. Hence, we have a well-defined and natural map from $\D_{\bar\alpha}$
to $\T$ which turns out to be the collapse of each ultrametric locus on a vertex
in the tree. This map is a covariant quasi-isometry:

\begin{theorem}
 For any $v\in \D_{\bar\alpha}$, there is a unique point $p:=\pi_{\bar\alpha}(v)\in \T$ such that the stabilizer of 
$p$ in $\PGL(2,\Q_p)$ contains the  stabilizer of $v$ in $\PGL(2,\Q_p)$.
The projection $\pi_{\bar\alpha}\: : \: \D_{\bar\alpha}\to \T$ defines a
quasi-isometry covariant for the action of $\PGL(2,\Q_p)$.

Moreover, $\pi_{\bar\alpha}$ induces a bijection between the set of long lines
in $\D_{\bar\alpha}$ and the set of geodesics in $\T$.
\end{theorem}

\begin{remark} Before going on with the proof, let us point out the similarity
with the case of the triangle (see section \ref{sec:genHilbdist}): in both
cases, the projection collapsing the ultrametric loci maps our
convex to a discretized version of their real counterpart. For the triangle, the
vertices of an hexagonal net reflected the hexagonal norm on the plane, and here
the vertices of a tree look very much like a discretized hyperbolic disc.
\end{remark}

\begin{proof}
I have no other proof of the first point than a direct computation: choose some $\alpha \in \bar\alpha$ of valuation $0$. The stabilizer $\Stab(v_\alpha)$ of $v_\alpha=(\alpha,0,1)$ in $\GL(2,\Q_p)$ is composed of the elements $g$ of the form: $$\begin{pmatrix} a & \alpha \epsilon c \\ c & \alpha a \end{pmatrix},$$ where $\epsilon=\det(g)=\pm 1$ and $a$ and $c$ are tied by the relation $c^2=\frac{1}{\alpha}(\epsilon-a^2)$.

As $-\alpha$ is not a square, it yields that $a\in \Z_p$ and then $g \in \PGL(2,\Z_p)$.

Now take an element of $h \in \PGL(2,\Q_p)$, represented by a matrix $$\begin{pmatrix} x & y \\ z & \alpha t \end{pmatrix}$$  in $\GL(2,\Q_p)$ of determinant $D$ verifying $|D|=1$ or $p$. The assumption $h\Stab(v_\alpha)h^{-1}\subset \PGL(2,\Q_p)$ gives us the system of conditions:
\begin{eqnarray}
a(tx - \epsilon yz)+c(ty+\alpha\epsilon yz) & \in & D\Z_p\\
axy(\epsilon - 1) -c(y^2+\alpha\epsilon x^2) & \in & D\Z_p\\
atz(-\epsilon + 1) +c(t^2+\alpha\epsilon z^2) & \in & D\Z_p\\
a(-zy+\epsilon tx) -c(ty+\alpha\epsilon zx) & \in & D\Z_p 
\end{eqnarray}
It should be verified for every $\epsilon=\pm 1$ and $a$, $c$ tied by the relation $c^2=\frac{1}{\alpha}(\epsilon-a^2)$.
We deduce that $x^2$, $y^2$, $z^2$, $t^2$, $tx$, $yz$, $ty$... belong to $D\Z_p$. So $|D|=1$ and $x$, $y$, $z$ and $t$ belong to $\Z_p$. We conclude that $h$ belongs to $\PGL(2,\Z_p)$.

This proves that the only maximal compact subgroup of $\PGL(2,\Q_p)$ containing $\Stab(v_\alpha)$ is $\PGL(2,\Z_p)$. As all the stabilizers are conjugated, the first point is proven.

\medskip

The covariance is clear. So it is enough to understand $\pi_{\bar\alpha}$ along a long line, using the transitivity of $\PGL(2,\Q_p)$. One easily identifies the projection along the long line between $(1,0,0)$ and $(0,0,1)$:
$$\pi_{\bar\alpha} (\alpha x^2,0,1) = \Z_p.(x,0)\oplus \Z_p.(0,1).$$

Hence if $v$ and $v'$ are at distance $d$ in $\D_{\bar\alpha}$, their projections $\pi_{\bar\alpha} (v)$ and $\pi_{\bar\alpha} (v)$ are at distance $E(\frac{d}{2})$. The projection is a quasi-isometry. It is nothing else than the collapse of the ultrametric loci in $\D_{\bar\alpha}$ to points in the tree. 

\medskip

The projection $\pi_{\bar\alpha}$ extends to the bijection $K_{\bar\alpha}(a^2,ab,b^2)\mapsto [a:b]$ from the isotropic cone $\C_{\bar\alpha}$ to $\P_1(\Q_p)$. As a long line (above) or a geodesic (in the tree) is uniquely defined by its ends in (respectively) $\C_{\bar\alpha}$ and $\P_1(\Q_p)$, the projection $\pi_{\bar\alpha}$ defines a covariant bijection between the long lines and the geodesics.
\end{proof}

\begin{remark}
 We won't be precise, but when the valuation of ${\bar\alpha}$ is odd and $-1$ is not a square, then each point in $\D_{\bar\alpha}$ has a well-defined projection to an \emph{edge} of $\T$. It explains why $\PSL(2,\Q_p)$ may act transitively on these $\D_{\bar\alpha}$.
\end{remark}

\section{The group of automorphisms}\label{sec:carac}

We show here that a transformation of $\D_{\bar\alpha}$ that preserves the multiplicative Hilbert distance is projective ; i.e. it  belongs to $\Isom(\D_{\bar\alpha})$. It does not hold under the weaker hypothesis of preserving the Hilbert distance because the latter lacks precision. For example, the Hilbert distance by itself does not allow to define long lines as ``geodesics'', whereas the multiplicative Hilbert distance does.

Throughout this section, $\bar\alpha$ is a fixed element in $\bk^*/(\bk^*)^2$, represented by some $\alpha\in\bk^*$.

\subsection{Multiplicative Hilbert distance}

Recall that we defined the notion of multiplicative Hilbert distance (definition \ref{def:multdist}). This notion will be enough to characterize in the following the action of $\PGL(2,\bk)$.
Of course the multiplicative Hilbert distance is invariant under the action of $\PGL(2,\bk)$.

The action of $\PGL(2,\bk)$ becomes transitive on the pair of points at equal multiplicative distance:
\begin{lemma}\label{lem:transit}
 For any $x$ in $\bk^*$, the action of $\PGL(2,\bk)$ on the set $\{(v,v') \in \D_{\bar\alpha}^2\textrm{ such that }D(v,v')=\{x,x^{-1}\}\}$ is transitive
\end{lemma}

\begin{proof}
 Fix a pair $(v,v')$ of points such that $D(v,v')=\{x,x^{-1}\}$. The group $\PGL(2,\bk)$ acts transitively on the flags (i.e. a point in a long line) by lemma \ref{prop:transit}, hence it sends $v$ on $v_\alpha=K_{\bar\alpha}(\alpha,0,1)$, and $v'$ on a point of the form $K_{\bar\alpha}(\alpha y^2, 0 ,1)$. As $\PGL(2,\bk)$ preserves the multiplicative distance, we have either $y=x$ or $y=x^{-1}$. The action of $\begin{pmatrix} 0 &-x \\ \frac{1}{x}&0\end{pmatrix}$ fixes $v_\alpha$ and maps $K_{\bar\alpha}(\alpha x^2, 0 ,1)$ to $K_{\bar\alpha}(\alpha,0, x^2)=K_{\bar\alpha}(\alpha x^{-2}, 0 ,1)$.

So the action of some element in $\PGL(2,\bk)$ maps the points $v$ and $v'$ to $K_{\bar\alpha}(\alpha,0,1)$ and $K_{\bar\alpha}(\alpha x^2, 0 ,1)$, proving the transitivity.
\end{proof}

We define now the notion of automorphism of $\D_{\bar\alpha}$ by asking that it preserves the multiplicative Hilbert distance. Indeed it seems reasonable to consider that this multiplicative distance is a natural invariant for two points in $\D_{\bar\alpha}$ lying on a long line. So a transformation of $\D_{\bar\alpha}$ which preserves its convex structure should preserve the multiplicative Hilbert distance:
\begin{definition}
 Let $T$ be a transformation of $\D_{\bar\alpha}$. It is an automorphism of $\D_{\bar\alpha}$ if it preserves the multiplicative distance, i.e.:
\begin{itemize}
 \item if $v$ and $v'$ lie on a long line, so do $T(v)$ and $T(v')$.
 \item for any $v$ and $v'$ on a long line, we have $D(T(v),T(v'))=D(v,v')$.
\end{itemize}
\end{definition}

The next section shows that any automorphism is indeed given by an element of $\PGL(2,\bk)$.

\subsection{Every automorphism is projective}

\begin{theorem}\label{the:carac}
 Consider an automorphism  $T$ of $\D_{\bar\alpha}$.

Then $T$ is a projective transformation preserving the quadratic form $Q$, i.e. $T$ belongs to $\Isom(\D_{\bar\alpha})$.
\end{theorem}

The proof is the same as in the real case: up to the action of $\Isom(\D_{\bar\alpha})$, we may assume that $T$ has a pointwise fixed long line and another fixed point. We prove that the multiplicative distance to these fixed points characterize any point in $\D_{\bar\alpha}$. We will need the notion of circle:
\begin{definition}
 Let $v$ be a point in $\D_{\bar\alpha}$ and $x\in \bk^*$. The circle centered
at $v$ and of multiplicative radius $\{r,r^{-1}\}$ is the set:
$$C(v,r)=\{v' \in \D_{\bar\alpha}\textrm{ such that }D(v,v')=\{r,r^{-1}\}\} .$$
\end{definition}

\begin{remark}
 The multiplicative Hilbert distance is defined for points on a long line. So for any $v'$ in the circle $C(v,r)$, the line $(vv')$ is a long one.
\end{remark}
We then study intersections of circles:
\begin{lemma}[Two circles intersect in at most two points]\label{lem:2circles}
Consider two distinct points $v$ and $v'$ in $\D_{\bar\alpha}$ on a long line. Fix $r$ and $r'$ in $\bk^*\setminus\{1\}$.

Then there are at most two points in the intersection between the circles $C(v,r)$ and $C(v',r')$. If there are effectively two points, there is an involution $g\in \Isom(\D_{\bar\alpha})$ which fixes $v$ and $v'$ and permutes these two points.
\end{lemma}

\begin{proof}
 Once again, one may assume that we have $v=K_{\bar\alpha}(\alpha,0,1)$ and $v'=K_{\bar\alpha}(\alpha x^2,0,1)$ and the proof is a calculus. The idea is that both circles are curves of degree $2$, hence may intersect in at most two points. 
 
 The circle $C(v,r)$ is  the orbit of $K_{\bar\alpha}(\alpha r^2, 0, -1)$ under $\Stab(v)$. And one may similarly see the circle $C(v',r')$ as the orbit:
$$\begin{pmatrix} x^2&0&0\\0&x&0\\0&0&1\end{pmatrix} \Stab(v)\,.\,K_{\bar\alpha}(\alpha r'^2, 0, 1).$$ We have already described $\Stab(v)$:
$$\Stab(v)=\left\{\Ad\begin{pmatrix} a & -\alpha \epsilon c \\ c &\epsilon a\end{pmatrix}\textrm{ for }\epsilon=\pm 1\textrm{ and }\alpha c^2=\epsilon-a^2\right\}.$$ Note that in the previous description one may transform $(a,c)$ in $(-a,-c)$ without changing the element in $\Stab(v)$. Hence the elements of $C(v,r)$ have coordinates:
$$
\left(\alpha \epsilon a^2r^2+(1-\epsilon a^2) ; (r^2-1)\epsilon ac; (1-\epsilon a^2)r^2+\epsilon a^2\right),
$$
 where $\epsilon=\pm 1$ and $\alpha c^2=\epsilon-a^2$.
For $C(v',r')$ we get:
$$
\left(x^2(\alpha \epsilon'a'^2r'^2+(1-\epsilon'a'^2)) ; (r'^2-1)\epsilon'a'c'; ((1-\epsilon'a'^2)r'^2+\epsilon'a'^2)\right)$$
where $\epsilon'=\pm 1$ and $\alpha c'^2=\epsilon'-a'^2$.
As $v$ and $v'$ are distinct, we have $x\neq 1$. Hence the equality of the first and the third coefficients gives a linear system in the unknowns $\epsilon a^2$ and $\epsilon' a'^2$ which has at most one solution. Using the equality between the second coefficient, there are at most two solutions $(r,s,t)$ and $(r,-s,t)$. Hence these two solutions are mapped one onto the other by the matrix:
$$\begin{pmatrix} 1&0&0\\0&-1&0\\0&0&1\end{pmatrix} \in \Stab(v)\cap \Stab(v').$$
\end{proof}

We are now able to see that three multiplicative distances are enough to define a point in $\D_{\bar\alpha}$:
\begin{lemma}[A point is defined by three multiplicative distances]\label{lem:3points}
Let $v_1$, $v_2$ and $v_3$ be three distinct points in $\D_{\bar\alpha}$, any two of them lying on a long line, but the three of them not lying on the same long line.

Then, for any $r_1$, $r_2$ and $r_3$ in $\bk^*$ there is at most one point in $\D_{\bar\alpha}$ at multiplicative distance $\{r_i,r_i^{-1}\}$ of the point $v_i$ for $i=1$, $2$ and $3$.
\end{lemma}

\begin{proof}
 Fix the $v_i$'s and $r_i$'s. Suppose by contradiction that two points $v\neq v'$ in $\D_{\bar\alpha}$ are at multiplicative distances $\{r_i,r_i^{-1}\}$ of the point $v_i$ for $i=1$, $2$ and $3$. One may assume that all the $r_i$ are different of $1$, because if $r_i=1$, the only possible solution is $v_i$.

One of the $v_i$, say $v_2$, does not belong to the line $(vv')$, because the $v_i$'s do not belong to the same line.
Using the previous lemma, for any $i\neq j$ there is an involution $t_{ij}$ in $\Isom(\D_{\bar\alpha})$ fixing $v_i$ and $v_j$ such that $t_{ij}(v)=v'$. Hence $t_{12}t_{23}$ fixes the three distinct points $v$, $v'$ and $v_2$ which are not on the same line. Hence $t_{12}t_{23}$ is a projective transformation of $\bk^3$ which fixes pointwise the two distinct projective lines $(vv')$ and $(vv_2)$. Hence it is an homothety. We get $t_{12}=t_{23}$.

It implies that the lines $(v_1v_2)$ and  $(v_2v_3)$ are the same, which contradicts the assumption that the $v_i$'s do not belong to the same line.
\end{proof}

The proof of the theorem follows:
\begin{proof}[Proof of theorem \ref{the:carac}]
Consider an automorphism $T$. Consider $3$ points $v_1$, $v_2$ and $v_3$, two of them defining a long line, but the three of them not belonging to the same line.

 The group $\Isom(\D_{\bar\alpha})$ acts transitively on the couple of points at same multiplicative distance (lemma \ref{lem:transit}). So there is an element $g_1$ in $\Isom(\D_{\bar\alpha})$ such that $g_1T$ fixes $v_1$ and $v_2$. As $g_1T$ is still an automorphism, it sends $v_3$ to one of the at most two points at distance $D(v_1,v_3)$ of $v_1$ and $D(v_2,v_3)$ of $v_2$. One then may choose, using lemma \ref{lem:2circles}, some $g_2$ in  $\Isom(\D_{\bar\alpha})$ such that $g_2g_1T$ fixes $v_1$, $v_2$, and $v_3$. By the lemma \ref{lem:3points}, $g_2g_1T$ is the identity.

It proves that $T$ is a projective map, i.e. it belongs to $\Isom(\D_{\bar\alpha})$.
\end{proof}

\section*{Conclusion}

It seems to the author that the construction described along this paper raises numerous questions. The notions of $p$-adic convexity and of Hilbert distance may not be suitably defined her ; but it would at least be interesting to test it on other examples and to see how rich $p$-adic Hilbert geometries are.
As for the hyperbolic discs, their existence gives a geometric object whose transformation group is $\PGL(2,\bk)$. They also deserve further studies. One may wonder if some problems or applications around the hyperbolic disc may be given a $p$-adic analogue. Here are three questions that seem worth exploring to me.

First of all, as mentioned in the introduction, the automorphisms of the tree form a huge group. We proved that if such an automorphism comes from an automorphism of the discs, then it acts as an element of $\PGL(2,\bk)$. But the geometry of the discs gives us a lot of invariants. For example, we get a notion of pencil of geodesics: a pencil of geodesics in the tree is (the projection of) a set of lines passing through a point in the disc above. What does  the group of automorphisms of the tree which map pencils to pencils looks like? And there are other notions to study, e.g. orthogonality of long lines.

Secondly, one is tempted to see a lattice in $\PGL(2,\bk)$ as the fundamental group of the quotient of hyperbolic discs by its action. Can we describe the $p$-adic surfaces that one obtains by this construction? In other words, what are the surfaces \emph{uniformized} by the $p$-adic hyperbolic discs ?

Eventually, and that was the starting point of this work, the notion of convexity in the real projective plane (and projective spaces of higher dimension) is the very starting point of the theory of divisible convex sets and relates with the Hitchin component of representations of a surface group (see \cite{benoist} for a survey on divisible convex sets). Though we lack the mere notion of connectedness, are we able to develop an analogue of some parts of these beautiful theories ?

\section*{Annex : Proof of fact \ref{fact:Iwazawa}}

Recall the setting of our Iwazawa
decomposition of the group $\SO(Q)$. Let $Q'$ be the quadratic form
$x_1^2+\ldots +x_{n-1}^2$. Consider the three following subgroups of
$\SL(n+1,\bk)$:
\begin{itemize}
 \item $N^+=\begin{pmatrix} 1 & 2 {}^t w A & Q'(w) \\ 0 & A & w \\ 0 & 0 & 1
\end{pmatrix}$ for $A \in \SO(Q')$ and $w\in \bk^{n-1}$.
 \item $N^-=\begin{pmatrix} 1 & 0 & 0 \\ v & B & 0 \\ Q'(v) &2{}^t v B & 1
\end{pmatrix}$ for $B \in \SO(Q')$ and $v\in \bk^{n-1}$.
 \item $H=\begin{pmatrix} x & 0 & 0 \\ 0 & Id & 0 \\ 0&0&\frac{1}{x}
\end{pmatrix}$ ($t\in\bk^*$).
\end{itemize}
The three following facts may be proven with geometric considerations in the
real case. But elementary linear algebra leads to the same conclusion and works
on any field.
\begin{fact}
\begin{itemize}
\item All three are subgroups of $\SO(Q)$ and $H$ normalizes both $N^+$ and
$N^-$.
\item The subgroup $N^+$ is the stabilizer of $v_0=\begin{pmatrix} 1\\ 0\\
\vdots \\ 0\end{pmatrix}$ in $\SO(Q)$.
\item The group $\SO(Q)$ decomposes as the product $N^- H N^+$.
\end{itemize}
\end{fact}

\begin{proof}
Let $P$ be the matrix representing the quadratic form $2Q$ in the standard basis
of $\bk^{n+1}$: $$P=\begin{pmatrix} 0 & 0 & 1\\ 0& -2Id & 0\\ 1
&0&0\end{pmatrix},$$ such that $Q(v)=\frac{1}{2} {}^t v P v$ for any $v \in
\bk^{n+1}$. The group $\SO(Q)$ is the group of matrices $M$ in $\SL(n+1,k)$
verifying ${}^tM P M=P$. 

It is clear that $H$ preserves $Q$ and normalizes $N^-$ and $N^+$.

Let us prove that $N^+$ is the stabilizer of $v_0$ in $\SO(Q)$ (it will also
prove it is indeed a subgroup of $\SO(Q)$ !): a matrix in $\SO(Q)$ stabilizing
the vector $v_0$ also stabilizes its orthogonal. The latter is the subspace
$\bk^n\times\{0\}$ of vectors having the last coordinate equals to $0$. Hence a
matrix $M$ stabilizing $v_0$ has a first column $=\begin{pmatrix} 1\\ 0\\ \vdots
\\0\end{pmatrix}$ and the last line $(0,\ldots,0,x)$ for some $x\in \bk$. In
other words, one may write: $$M=\begin{pmatrix} 1 & {}^t v & y \\ 0 & A & w \\ 0
& 0 & x \end{pmatrix}$$ where $A$ is a $(n-1)\times (n-1)$-matrix, and $v$, $w$
belong to $\bk^{n-1}$. A straightforward computation gives:
$${}^t M P M = \begin{pmatrix} 0 & 0 & x \\ 0 & -2{}^tAA & xv -2{}^t Aw \\ x &
{}^t(xv-2{}^tAw & 2xy - 2{}^tw w \end{pmatrix}.$$
The equality ${}^t M P M = P$ leads to, successively, $x=1$, $v=2{}^t Aw$, ${}^t
A A=Id$, $y={}^tw w$. The last two may be translated into: $A \in \SO(Q')$ and
$y=Q'(w)$, which proves the fact.
One may prove along the same lines that $N^-$ is the stabilizer of the vector
$\begin{pmatrix} 0\\ \vdots\\ 0 \\ 1\end{pmatrix}$.

For the third claim, one shows that the product $N^- H=H N^-$ send $v_0$ to any
isotropic vector $v=\begin{pmatrix}x_0 \\ \vdots\\ x_n\end{pmatrix}\in \mathcal
C\setminus\{ 0\}$. Let $v'=\begin{pmatrix}x_1 \\ \vdots\\ x_{n-1}\end{pmatrix}$
and define the following matrix of $N^-$: $$n^-=\begin{pmatrix} 1 & 0 & 0 \\ v'
& Id & 0 \\ Q'(v') &2{}^t v' & 1 \end{pmatrix}.$$
Then $n^- (v_0)= \begin{pmatrix} 1 \\v'\\ Q'(v')\end{pmatrix}$. Now $v$ is
isotropic, i.e. $Q(v)=x_0x_n-Q'(v')=0$, so $x_0x_n=Q'(v')$. It remains to
consider the following matrix of $H$: 
$$h=\begin{pmatrix} x_0 & 0 & 0 \\ 0 & Id & 0 \\ 0&0&\frac{1}{x_0}
\end{pmatrix}.$$ It verifies $hn^- (v_0)=v$.

We may conclude: let $g$ be an element of $\SO(Q)$ and choose two elements $n^-$
and $h$ such that $n^- h (v_0)=g(v_0)$. Then $(n^- h)^{-1} g$ fixes $v_0$. So it
belongs to $N^+$ and $g$ belongs to $n^- h N^+$.
\end{proof}

\bibliographystyle{amsalpha}
\bibliography{yapahp}

\end{document}